\documentclass[12pt]{article}
\usepackage[centertags]{amsmath}
\usepackage{amsfonts}
\usepackage{amssymb}
\usepackage{amsthm}
\input{amssym.def}
\usepackage{graphicx}

\numberwithin{equation}{section}

\newtheorem{Definition}{Definition}[section]
\newtheorem{theorem}[Definition]{Theorem}

\newtheorem{corollary}[Definition]{Corollary}

\begin{document}
\title{\Large \bf  On some properties of enhanced power graph}
\author{ Sudip Bera and  A. K. Bhuniya  }
\date{}
\maketitle

\begin{center}
Department of Mathematics, Visva-Bharati, Santiniketan-731235, India. \\
sudipbera517@gmail.com,  anjankbhuniya@gmail.com
\end{center}

\begin{abstract}
 Given a group $G$, the enhanced power graph of $G$ denoted by $\mathcal{G}_e(G)$,  is the graph with vertex set $G$ and two distinct vertices $x, y$ are edge connected  in $\mathcal{G}_e(G)$ if there exists $z\in G $ such that $x=z^m$ and $  y=z^n $, for some $m,  n\in \mathbb{N}$. In this article,  we characterize the enhanced power graph  $\mathcal{G}_e(G)$ of $G$. The graph $\mathcal{G}_e(G)$ is complete if and only if $G$ is cyclic, and $\mathcal{G}_e(G)$ is Eulerian if  and only if $|G|$ is odd. We classify all abelian groups  and also all non-abelian $p-$groups $G$ for which $\mathcal{G}_e(G)$ satisfies the cone property.
\end{abstract}


\section{Introduction}
Given an algebraic structure $S$, we can associate this algebraic structure $S$ to a directed or undirected graph in different ways\cite{sen}, \cite{r}, \cite{sing}. To study algebraic structures using graph theory,  different graph has been formulated, namely power graph of semigroup \cite{sen}, strong power graph of group \cite{sing}, normal subgroup based power graph of group \cite{BB1},   ideal based zero divisor graph of a ring \cite{r} etc. The directed power graph of a semigroup was defined by Kelarev and Quinn \cite{K}. Then Chakraborty et. al defined the undirected power graph $\mathcal{G}(S)$ of a semigroup $S$, where the vertex set of the graph is $S$ and two distinct vertices $x, y$ are adjacent if either $x=y^m$ or $y=x^n$ for some $m, n\in\mathbb{N}$. Many researcher generalized the undirected power graph in different ways. Aalipour et. al \cite{pj} defined the enhanced power graph of a group $G$. The enhanced power graph of a group $G$,  is denoted by $\mathcal{G}_e(G)$,  is the graph whose vertex set is the group $G$ and two distinct vertices $x, y$ are edge connected if  there exists $z \in G$ such that $x=z^m$ and $y=z^n$ for some $m, n \in \mathbb{N}$. In Section $2$, some basic structure have been studied. From the definition it follows that $\mathcal{G}_e(G)$ is complete if $G$ is cyclic. Here we show that the converse also holds. The graph $\mathcal{G}_e(G)$ contains a cycle if and only if $o(a)\geq 3$, for some $a\in G$. As a consequence of this result, all finite group $G$ have been characterized such that $\mathcal{G}_e(G)$ is bipartite, tree, and star graph. In Section $3$, we classify all abelian groups $G$ such that $\mathcal{G}_e(G)$ satisfies the cone property (details later). If $G$ is a non-abelian group $p-$group, then $\mathcal{G}_e(G)$ satisfies the cone property if and only if $G$ is a generalized quarternion group. Complete characterization of the groups $G$ such that the graph $\mathcal{G}_e(G)$ is planar or  Eulerian have been done in Section $4$. Note that the identity element $e$ of the group $G$ is adjacent to every other vertex in $\mathcal{G}_e(G)$. Section $5$ is devoted to characterize the deleted enhanced power graph  $\mathcal{G}_e^*(G)$, a subgraph of $\mathcal{G}_e(G)$ obtained by deleting the vertex $e$.
\section{Definition and basic structure}
Throughout this article $G$ stand for a finite group. We denote $o(x)$ to be the order of an element $x$ in $G$, $\pi_e(G)=\{ o(x):x \in G\}$, $\pi(n)=\{ p\in \mathbb{N} : p|n$ is a prime,$n \in \mathbb{N}\}$  and $\pi(G)=\pi(|G|)$. Let $\mu(G)\subset \pi_e(G)$ be the set of all maximal element of $\pi_e(G)$ under the divisibility relation.

 For any vertex $a$ of the enhanced power graph $\mathcal{G}_e(G)$ we have $a, e \in <a>$.  So the graph $\mathcal{G}_e(G)$ is connected. Let $a\in G$. We denote  $\mathcal{G}en(a)$ to be the set of all generators of the cyclic subgroup  $<a>$ of $G$.  Then $G=\bigcup_{a \in G}\mathcal{G}en(a)$. Now we  show that the vertices  in  $\mathcal{G}en(a)$ form a clique in $\mathcal{G}_e(G)$ for every $a \in G$. In fact,  vertices $x, y$ in $\mathcal{G}en(a)$ we have $x=a^m$ and $y=a^n$ for some $m, n \in \mathbb{N}$. And hence $x\sim y$. Suppose that $\mathcal{G}en(a)\neq \mathcal{G}en(b)$. Let $x\in \mathcal{G}en(a)$, $y \in\mathcal{G}en(b)$ with $x \sim y$. Then there exists $z \in G$ such that $x, y \in <z>$. Now for any $x_1 \in \mathcal{G}en(a)$ and any $y_1 \in \mathcal{G}en(b)$ we have $<x>=<x_1>\subset <z>$ and $<y>=<y_1>\subset <z>$. Hence all the vertices in $\mathcal{G}en(a)$ are adjacent with all the vertices in $\mathcal{G}en(b)$.
\begin{theorem}
Let $a, b\in G$ with $o(a)=o(b)$ and $<a>\neq <b>$. Then none of the vertices  in $\mathcal{G}en(a)$ is adjacent with the vertices in $\mathcal{G}en(b)$ in $\mathcal{G}_e(G)$.
\end{theorem}
\begin{proof}
Suppose that $x\in \mathcal{G}en(a)$ and $y\in \mathcal{G}en(b)$ with $x\sim y$ in $\mathcal{G}_e(G)$. Then there exists $z\in G$ such that $x, y \in <z>$. Now $o(a)=o(b)$ implies that $o(x)=o(y)$. So the cyclic subgroup $<z>$ contains two distinct  subgroups of same order. This is  a contradiction. Hence the result follows.
\end{proof}
\begin{theorem}
 The enhanced power graph $\mathcal{G}_e(G)$ of the group $G$ contains a cycle if and only if $o(a)\geq 3$, for some $a\in G$.
\end{theorem}
\begin{proof}
First suppose that $\pi_e(G)\subset \{ 1, 2\}$. Then for every  $a\in G $, $\mathcal{G}en(a)$ contains exactly one element. Now by the discussion above of the Theorem $2.1$ and by the Theorem $2.1$,  the enhanced power graph $\mathcal{G}_e(G)$ has no cycle.

Conversely, suppose that $a\in G$ such that $o(a)\geq 3$. Then $|\mathcal{G}en(a)|\geq 2$. So the vertices in $\mathcal{G}en(a)$ with the identity form a cycle. Hence the result.
\end{proof}
If $G$ is a finite group such that $o(x)=2$ for every non-identity element $x$ of $G$, then $G$ is abelian and $G\cong \mathbb{Z}_2\times \mathbb{Z}_2 \times \cdots \times \mathbb{Z}_2$. Also a connected graph $G$ is tree if and only if  it has no cycle.
\begin{corollary}
Let $G$ be a group. Then the following conditions are equivalent.
\begin{enumerate}
\item
 $\mathcal{G}_e(G)$ is bipartite;
 \item
 $\mathcal{G}_e(G)$ is tree;
 \item
$G\cong \mathbb{Z}_2\times \mathbb{Z}_2 \times \cdots \times \mathbb{Z}_2$;
 \item
 $\mathcal{G}_e(G)$ is a star graph.
\end{enumerate}
\end{corollary}
\begin{proof}
Equivalences of $(1), (2)$ and $(3)$ follows directly from Theorem $2.2$.

$(3)\Rightarrow (4)$

 Let $G\cong \mathbb{Z}_2\times \mathbb{Z}_2\times \mathbb{Z}_2 \times\cdots \times \mathbb{Z}_2$. Then each non-identity element of $G$ is of order $2$. This implies that $\mathcal{G}en(a)=\{a\}$, for each non-identity element $a$ of the group $G$, and  $\mathcal{G}_e(G)$ is a star graph by Theorem $2.1$.

 $(4)\Rightarrow (3)$

  Suppose that $\mathcal{G}_e(G)$ is a star graph. The identity element $e$ is adjacent to all other vertices in   $\mathcal{G}_e(G)$. Hence no two non-identity elements adjacent in $\mathcal{G}_e(G)$. This is possible only when order of each non-identity element of $G$ is $2$. Hence $G\cong \mathbb{Z}_2\times \mathbb{Z}_2\times \mathbb{Z}_2 \times\cdots \times \mathbb{Z}_2$.
\end{proof}
A graph $\Gamma$ is complete if any two vertices of the graph $\Gamma$ are adjacent. Now we find the condition that the enhanced power graph $\mathcal{G}_e(G)$ is complete.
 \begin{theorem}
  The enhanced power graph $\mathcal{G}_e(G)$ of the group $G$ is complete if and only if $G$ is cyclic.
  \end{theorem}
 \begin{proof}
 Suppose that the group $G$ is cyclic. Then $G=<a>$, for some $a\in G$. Let $x$ and $y$ be two distinct vertices  of $\mathcal{G}_e(G)$. Then $x, y\in G=<a>$, which implies that $x\sim y$ in $\mathcal{G}_e(G)$. Hence the graph $\mathcal{G}_e(G)$ is complete.

 Conversely, suppose that the graph  $\mathcal{G}_e(G)$ is complete. If possible $G$ is non-cyclic.  Let $x$ be an element of $G$ such that $o(x)$ is maximum  and $H=<x>$. Now $G\neq H$ implies that there exists $y\in G\setminus H$. Since the graph $\mathcal{G}_e(G)$ is complete, so $x\sim y$ in $\mathcal{G}_e(G)$. Then there exists $z \in G$ such that $x, y \in <z>$. Since $o(x)$ is maximum, we have $<x>=<z>$, which leads to contradiction that  $y \in <x>$. Hence $G$ is cyclic group.
 \end{proof}

 \section{Cone property of $\mathcal{G}_e(G)$}
 In this section we classify all abelian groups and non-abelian $p-$groups $G$ such that $\mathcal{G}_e(G)$ has a vertex other than the identity element, which is adjacent to every vertex. In the context of power graphs, this property  has been studied in \cite{Cameran}, \cite{Cameran 2}. For the sake of smoothness in discussion, we call such a vertex a cone vertex and a graph having a cone vertex is said to satisfy the cone property.
\begin{theorem}
Let $G$ be a finite group and $n\in \mathbb{N}$. If $gcd(|G|, n)=1$, then the enhanced power graph $\mathcal{G}_e({G\times \mathbb{Z}_n})$ has a cone vertex.
\end{theorem}
\begin{proof}
We show that $(e, a)$ is a cone vertex of the graph $\mathcal{G}_e(G)$, where $a$ is a generator of $\mathbb{Z}_n$. Let $(g, x)\in G\times\mathbb{Z}_n$ and $x=a^m$. Then $gcd(o(g), o(a))=1$ implies that $(g, 0), (e, a)\in <(g, a)>$ and so $(g, x)=(g, 0)(e, a)^m\in <(g, a)>$. Thus $(g, x), (e, a)\in <(g, a)>$ and hence $(g, x)\sim (e, a)$ in $\mathcal{G}_e(G)$.
\end{proof}
Now we characterize all finite abelian groups $G$ for which  the enhanced power graph $\mathcal{G}_e(G)$ has a cone vertex. Let us fix a notation $[m]=\{ 1, 2, 3, \cdots, m\}$
\begin{theorem}
Let $G$ be a finite abelian group. Then $\mathcal{G}_e(G)$ has a cone vertex if and only if $G$ has a cyclic Sylow subgroup.
\end{theorem}
\begin{proof}
Suppose that $G$ has a cyclic Sylow subgroup $P$ and $|P|=p^m, p$ is a prime and $m \in \mathbb{N}$ . Then by  the fundamental theory  of finite  abelian group $G=H\times \mathbb{Z}_{p^m}$, where $m \in \mathbb{N}, H$ is subgroup of the  group $G$, $p$ is prime and $gcd(|H|, p^m)=1$. Then by the previous theorem $\mathcal{G}_e(G)$ satisfies the cone property.

Conversely, suppose that the group $G$ has no cyclic Sylow subgroups.  Then we show that $\mathcal{G}_e(G)$ does not satisfy the cone property. First we note that under the assumption $G$ is of the form $$G=\mathbb{Z}_{p^{t_{11}}_1}\times \mathbb{Z}_{p^{t_{12}}_1}\times\mathbb{Z}_{p^{t_{13}}_1}\times\cdots\times\mathbb{Z}_{p^{t_{1k_1}}_1}\times \mathbb{Z}_{p^{t_{21}}_2}\times\mathbb{Z}_{p^{t_{22}}_2}\times\cdots\times\mathbb{Z}_{p^{t_{2k_2}}_2}\times\cdots\times
\mathbb{Z}_{p^{t_{r1}}_r}\times\mathbb{Z}_{p^{t_{r2}}_r}\times\cdots\times\mathbb{Z}_{p^{t_{rk_r}}_r}$$, where $k_i\geq2 $ and $t_{i1}\leq t_{i2}\leq\cdots\leq t_{ik_i} $, for all $i\in [r]$. Suppose $\mathcal{G}_e(G)$ has a cone element $v=(x_{p^{t_{11}}_1}, x_{p^{t_{12}}_1}, \cdots x_{p^{t_{1k_1}}_1}, x_{p^{t_{21}}_2}, \cdots, x_{p^{t_{2k_2}}_2}, \cdots, x_{p^{t_{1r}}_r}, x_{p^{t_{2r}}_r}, \cdots, x_{p^{t_{rk_r}}_r} )$. Now we claim that for each $i\in [r]$ and for each $j\in [k_i-1], x_{p^{t_{i1}}_i}=x_{p^{t_{i2}}_i}=\cdots=x_{p^{t_{ij}}_i}=0$. [Here $0$ actually means the additive identity of the group $\mathbb{Z}_{p^{t_{ij}}_i}$].  Consider the element $\acute{v}=(0, 0, \cdots, 0, g_{p^{t_{1k_1}}_1}, 0, 0, \cdots, g_{p^{t_{2k_2}}_2}, 0, \cdots, 0, g_{p^{t_{rk_r}}_r} )$, where  $g_{p^{t_{ik_i}}_i}$ are generators of the cyclic group $\mathbb{Z}_{p^{t_{ik_i}}_i}$ for each $i\in[r]$. Since $v$ is a cone element of the graph $\mathcal{G}_e(G)$, we have $v\sim \acute{v}$ in $\mathcal{G}_e(G)$. Now $\acute{v}$ is a maximum ordered element of the group $G$ implies that $v\in<\acute{v}>$.  So for each $i\in [r]$ and for each $j\in [k_i-1], x_{p^{t_{i1}}_i}=x_{p^{t_{i2}}_i}=\cdots=x_{p^{t_{ij}}_i}=0$, i. e. $v=(0, 0, \cdots, 0, x_{p^{t_{1k_1}}_1}, 0, 0, \cdots, 0, x_{p^{t_{2k_2}}_2}, 0, \cdots, 0, x_{p^{t_{rk_r}}_r} )$. Since $v$ is a non-identity element of the group $G$, atleast one of the $x_{p^{t_{ik_i}}_i}$ is non-zero. Without loss of generality we assume that $x_{p^{t_{1k_1}}_1}\neq 0$. Consider $v_1=(x, 0, 0, \cdots, 0, 0)$, where $x\in \mathbb{Z}_{p^{t_{ik_1}}_1}$ with $o(x)=p_1$. Then $p_1|o(v)$.  Now we show that $v$ is not adjacent to $v_1$.  Now $v\sim v_1$ implies that there exists a cyclic subgroup $C$ of $G$ such that $v, v_1\in C$. Then $o(v_1)=p_1$ and $p_1$ divides $o(v)$ implies that $v_1\in<v>$, which contradicts that $x\neq 0$. Hence the theorem.
\end{proof}

Now we turn our attention to the non-abelian groups. First note that, from the Theorem $3.1$ we get an infinite family of non-abelian groups which satisfy the cone property but none of these groups are $p-$groups. So the next natural question  that occurs is to classify all non-abelian $p-$groups $G$, whose enhanced power graph $\mathcal{G}_e(G)$ satisfy the cone property. The next theorem completely answers the question.
\begin{theorem}
Let $G$ be a non-abelian $p-$group. Then the enhanced power graph $\mathcal{G}_e(G)$ satisfies the cone property if and only if $G$ is generalized quarternion group.
\end{theorem}
\begin{proof}
First assume that $G$ is a generalized quarternion group. So $G$ has a unique minimal subgroup say $H=<x>, x \in G$. And clearly $x$ is of order $p$.  We show that $x$ is a cone element. Let $y\in G$. Then $o(y)=p^t, t \in \mathbb{N}$. Since $H$ is unique minimal subgroup we have $x\in<y>$. Hence $x\sim y$ in $\mathcal{G}_e(G)$.

Conversely, suppose that $\mathcal{G}_e(G)$ has a cone element $x$.  We show that,  every order $p$ element belongs to the group $<x>$. For, suppose that $v\in G$ is of order $p$. Since $v\sim x$, $v, x\in C$, $C$ cyclic and consequently $v\in <x>$. Hence $G$ has a unique subgroup of order $p$, i. e. $G$ has a unique minimal subgroup and hence  $G$ is isomorphic to a generalized quarternion group \cite{gorestin}.
\end{proof}
\begin{theorem}
Let $G$ be any simple group. Then $\mathcal{G}_e(G)$ does not satisfy the cone property.
\end{theorem}

\begin{proof}
Let, if possible $\mathcal{G}_e(G)$ satisfy the cone property. Suppose that, $v$ is a cone element. Let $o(v)= m$ and $p| m, p$ prime. We claim that, there exists a unique subgroup of order $p$. Let $x \in G$ such that $o(x)= p$. We show that $x \in <v>$. Since $v$ is cone element, $x\sim v$. So there exists cyclic subgroup $C$ such that $x,v \in C$. Hence $x \in <v>$. So there exists a unique subgroup of order $p$, which is normal in $G.$ Contradiction.
\end{proof}

We leave the problem of complete classification of non-abelian groups, whose enhanced power graphs satisfy the cone property as an interesting open problem.

\section{Eulerian and planar graph}
In this section we characterize the  groups $G$  such that the graph $\mathcal{G}_e(G)$ is Eulerian and planar.  A graph $\Gamma$ is called Eulerian if it has a closed trail containing all the vertices of $\Gamma$. An useful equivalent characterization of an Eulerian graph is that a graph $\Gamma$ is Eulerian if and only if every  vertex of $\Gamma$ is of even degree.
A graph $\Gamma$ is called  planer if it can be drawn in a plane so that  no two edges intersect. A graph is planer if and only if it does not contain a graph which is isomorphic to either of the graphs $K_{3, 3}$ and $K_5$.
\begin{theorem}
Let $G$ be a group. Then the enhanced power graph $\mathcal{G}_e(G)$ is planar if and only if $\pi_e(G)\subset \{ 1, 2, 3, 4\}$.
\end{theorem}
\begin{proof}
Suppose the graph $\mathcal{G}_e(G)$ is planar.  If $G$ has an element $x$ such that $o(x)\geq 5$. Then the vertices in $<a>$ forms a clique of size  $ \geq 5$ in $\mathcal{G}_e(G)$. Hence $\mathcal{G}_e(G)$ has a subgraph isomorphic to $K_5$ and so $\mathcal{G}_e(G)$ can not be planar.

Conversely, suppose that $\pi_e(G)\subset \{ 1, 2, 3, 4\}$. Let $a, b \in G$ such that $o(a)=2$ and $o(b)=3$. Since there are no element of order $6k, k\in \mathbb{N}$ in $G$, $a$ is not adjacent to $b$ in $\mathcal{G}_e(G)$. Similarly if $x_1, y_1\in G$ such that $o(x_1)=4$ and $o(y_1)=3$, then $x_1$ is not edge connected to $y_1$. Again any element $x$ of $G$ of order $4, e, x, x^2, x^3,$ form a complete subgraph $K_4$ of $\mathcal{G}_e(G)$. Now by the Proposition $2.1$ and by the above discussion $\mathcal{G}_e(G)$ is planar graph.
\end{proof}
\begin{theorem}
Let $G$ be a group of order $n$. Then the enhanced power graph $\mathcal{G}_e(G)$ is Eulerian if and only if $n$ is odd.
\end{theorem}
\begin{proof}
Suppose that the graph $\mathcal{G}_e(G)$ is Eulerian. Since the vertex $e$ is edge connected with every other vertices of the graph $\mathcal{G}_e(G)$, it follows that the degree of $e$ is $n-1$. Now $n-1$ is even implies that $n$ is odd.

Conversely assume that $n$ is odd. Then the degree of $e$ in $\mathcal{G}_e(G)$ is $n-1$ and so even. Now we show that the degree of every non-identity element $a$ is even. The vertex set of the enhanced power graph can be written as $V(\mathcal{G}_e(G))=\bigcup_{x\in G }\mathcal{G}en(x)$, where $\mathcal{G}en(x)$ is the collection of all generators of the cyclic subgroup $<x>$. Now it follows  from the discussion before Theorem $2.1$, that all the vertices of the graph $\mathcal{G}_e(G)$ in $\mathcal{G}en(x)$ form a clique and if $x\in \mathcal{G}en(a)$, $y \in\mathcal{G}en(b)$ with $x \sim y$ then all the vertices in $\mathcal{G}en(a)$ are adjacent to all the vertices in $\mathcal{G}en(b)$. Since $\mathcal{G}en(x)$ contains $\phi(o(x))$  vertices and every vertex is adjacent to $e$, so the degree of a vertex $a$ in the graph $\mathcal{G}_e(G)$ is of the form $(\phi(o(a))-1)+\phi(o(x_1))+\phi(o(x_2))+\cdots+\phi(o(x_m))+1=\phi(o(a))+\phi(o(x_1))+\phi(o(x_2))+\cdots+\phi(o(x_m))$. Now $n$ is odd implies that $o(x)$ is odd and so  $\phi(o(x))$ is even for all $x\in G$. Thus  the degree of every vertex of the graph $\mathcal{G}_e(G)$ is even. Hence the enhanced power graph is Eulerian.
\end{proof}
\section{The deleted enhanced power graph of a group}
In this section we consider the subgraph $\mathcal{G}^*_e(G)$ obtained by deleting the vertex $e$ from the graph $\mathcal{G}_e(G)$. We call  $\mathcal{G}^*_e(G)$ the deleted enhanced power graph. Since the vertex $e$ is adjacent to every other vertices in $\mathcal{G}_e(G)$ for every group $G$, so it is expected to get new information on  the interplay of the group theoretic  properties of $G$ with the graph theoretic properties on consideration of the deleted enhanced power graph $\mathcal{G}^*_e(G)$. For every group $G$, the graph $\mathcal{G}_e(G)$ is connected, but $\mathcal{G}^*_e(S_3)$ is not connected whereas $\mathcal{G}^*_e(\mathbb{Z}_6)$  is connected. We prove the following results in this section with the help of the paper \cite{Shi}.
\begin{theorem}
Let $G$ be a finite $p-$group. Then the  graph $\mathcal{G}^*_e(G)$ is connected if and only if $G$ has a unique minimal subgroup.
\end{theorem}
\begin{proof}
Suppose that  $G$ has a unique minimal subgroup,  say $H$.  Now $|G|=p^t$ implies that $|H|=p$ and $H=<x>$, where $p$ is  prime and $x\in G$.  Let $y\in V(\mathcal{G}^*_e(G) )$.
 Then $o(y)=p^t$, for some $t \in \mathbb{N}$. Now $o(x)$ divides $o(y)$ and $<x>$ is the unique minimal subgroup of $G$ implies that $<x>\subset<y>$. So $x\sim y$.  Hence the graph $\mathcal{G}^*_e(G)$ is connected.

Conversely suppose the graph $\mathcal{G}^*_e(G)$ is connected. If possible, assume that $G$ has two minimal subgroups  $<x>$ and $<y>$. Then $o(x)=o(y)=p$ implies that $x$ is not adjacent to $y$, by Theorem $2.1$.  Since $\mathcal{G}^*_e(G)$ is connected there exists a path say $x=x_0\sim x_1\sim x_2 \sim \cdots \sim x_n=y$ of least length between $x$ and $y$. Then $n\geq 2$.  Now $x\sim x_1$ and $x_1\sim x_2$ implies that there exists $z_1, z_2\in G$ such that $x, x_1 \in <z_1>$ and $x_1, x_2 \in <z_2>$. Again  the cyclic subgroup $<x>$ is minimal, so $<x>\subset<x_1>$. Now either $<x_1>\subset <x_2>$ or $<x_2>\subset <x_1>$. If $<x_1>\subset <x_2>$ then we have $x, x_2\in <x_2>$ i. e. $x\sim x_2$. Again if $<x_2>\subset <x_1>$ then $<x>\subset <x_1>$ and $<x_2>\subset <x_1>$ implies that $x, x_2\in <x_1>$. So $x\sim x_2$, a contradiction. Hence $G$ has a unique minimal subgroup.
\end{proof}
\begin{theorem}
Let  $|\pi(Z(G))|\geq 2$. Then the graph $\mathcal{G}^*_e(G)$ is connected.
\end{theorem}
\begin{proof}
The group $Z(G)$ is abelian, this implies that $|\mu(Z(G))|=1$. Let $\mu(Z(G))=t$ and  $g\in Z(G)$,  such that $o(g)=t$. Now we show that for each $x\in G\setminus \{e\}$  there is a path between $x$ and $g$.

Case1: Let $\pi(o(x))=\pi(t)$. Since $|\pi(t)|\geq 2$, it follows there are two distinct primes $p$ and $q$ in $\pi(t)$. Now $\pi(o(x))=\pi(t)$ implies that $o(x^r)=p$ and $o(g^s)=q$ for some natural numbers $r, s$. Then  $gcd(o(x^r, o(g^s)))=1$ and $g \in Z(G)$ implies that . Now $o(x^rg^s)=pq$ and $x^r, g^s \in <x^rg^s>$. Hence we have a path $x\sim x^r\sim g^s \sim g$ in $\mathcal{G}^*_e(G)$.

Case2: Let $\pi(o(x))\neq \pi(t)$. In this case there exists a prime $p\in \pi(o(x))\setminus \pi(t)$ or $p\in \pi(t)\setminus \pi(o(x))$.  Consider  the case $p\in \pi(o(x))\setminus \pi(t)$, the other case is similar. Let $q\neq p$ be a prime and $q \in \pi(t)$.Then similar as above there exists $r, s \in \mathbb{N}$ such that $o(x^r)=q$ and $o(g^s)=p$, and we have a path $x\sim x^r \sim g^s \sim g$.

Hence that graph $\mathcal{G}^*_e(G)$ is connected.
\end{proof}
\begin{theorem}
Let $|\pi(G)|\geq 2$ and $|Z(G)|=p^t$ for some prime $p\in \pi(G)$. Then the graph $\mathcal{G}^*_e(G)$ is connected if and only if for some non-central element $x$ of order $p$ there exists a non $p-$element $g$  such that $x\sim g$ in the graph $\mathcal{G}^*_e(G)$.
\end{theorem}
\begin{proof}
First suppose that the graph $\mathcal{G}^*_e(G)$ is connected. Let if possible,  there be  a  non-central $p-$element $x$ of order $p$ such that for every non $p-$element $g, x$ is not edge connected with $g$. Let $z\in Z(G)$ with $o(z)=p$.  Let there is a path $P: x=x_0\sim x_1 \sim x_2 \sim \cdots \sim x_k =z$ with minimal length. Let $x \sim z$, then there exists an element $z_1 \in G$ such that $x, z \in <z_1>$. Again $o(x)=p$ and $Z(G)$ is a $p-$group implies that $o(z)=p^t$ and $x\in <z>\subset Z(G)$, i. e. $x \in Z(G)$, a contradiction. So $x, z$ are not edge connected in the graph $\mathcal{G}^*_e(G)$.

We assume that $k >1$. Since $x\in c_1$ there exists $z_2\in G$ such that $x, c_1 \in <z_2>$. Now by assumption $c_1$ is a $p-$element of $G$. Again $o(x)=p$ and $x, c_1 \in <z_2>$ implies that $<x>\subset <g>$. Since $c_1\sim c_2$ there exists $z_3 \in G$ such that $c_1, c_2 \in <z_3>$. So $<x>\subset <c_1>$ and $c_1, c_2 \in <z_3>$ implies that $x, c_2 \in <z_3>$, i. e. $x \sim c_2$, a contradiction.

Conversely, for every non-central element $x$ of order $p$ there exists a non $p-$element $g$ such that $x \sim g$ in the graph $\mathcal{G}^*_e(G)$. We show that for each $g \in G\setminus Z(G)$ there is a path from $g$ to each non-trivial element of $Z(G)$.

Case1: First suppose that $g$ is not a $p-$element. Suppose $q\in \pi(G)\setminus \{ p\}$ such that $q$ divides $o(g)$ and $o(g^t)=q$, for some $t \in \mathbb{N}$. Let $z\in Z(G)$ and $o(z)=p^r, r \in \mathbb{N}$. Now proceeding as the first part we have $g\sim g^t \sim z$.

Case2: Let $g$ is a $p-$element. Then there is a positive integer $s$ such that $o(g^s)=p$. Let $x$ be a $\acute{p}-$element of $G$, then $x \sim g^s \sim g$. Now $x$ is not a $p-$element, so by case1 there is a path between $x$  and any non-trivial element of $Z(G)$. Hence there is a path $g$ and every non-trivial element of $Z(G)$. Now suppose that $g^s$ does not belong to $Z(G)$. Now by assumption there is a non $p-$element $x_1$ in $G$  such that $g^s \sim x_1$. Again $g\sim g^s \sim x_1$. So by case1 there is a path between $g$ and all non-trivial elements in $Z(G)$.
\end{proof}
\begin{theorem}
The following conditions are equivalent.
\begin{enumerate}
\item
$o(a)<4$ for every $a \in G$;
\item
$\mathcal{G}^*_e(G)$ has no cycle;
\item
$\mathcal{G}^*_e(G)$ is bipartite;
\item
$\mathcal{G}^*_e(G)$ is a forest.
\end{enumerate}
\end{theorem}
\begin{proof}
Equivalences  of $(2), (4)$ and $(3)$ are direct.

$(1)\Rightarrow (2)$: Suppose that $o(a)<4$ for every $a \in G$. Since every finite cyclic group has unique subgroup of a particular order, so two element of order $2$ can not be adjacent in $\mathcal{G}^*_e(G)$. Since $G$ has no element of order $6$, so an element of order $2$ can not be adjacent with an element of order $3$. Also it follows from Theorem $2.1$ that an element of order $3$ can not be adjacent with other two elements of order $3$. Hence $\mathcal{G}^*_e(G)$
can not have any cycle.

$(2) \Rightarrow (1)$: If $G$ has an element  $a $ such that $o(a)\geq 4$, then $a\sim a^2\sim a^3\sim a$ is a cycle in $\mathcal{G}^*_e(G)$.
\end{proof}

\noindent
\textbf{Open problem:}  Characterize all finite non-abelian groups $G$ such that $\mathcal{G}_e(G)$ has the cone property.

\noindent
\textbf{Acknowledgement:}  The authors acknowledge Mr.Sajal Kumar  Mukherjee for proposing the Theorems $3.1, 3.2, 3.3$.  The first author is partially  supported by UGC-JRF grant, India.

\bibliographystyle{amsplain}

\begin{thebibliography}{10}
\baselineskip 5mm


\bibitem{survey}
 J.  Abawajy,    A. V. Kelarev,  M.  Chowdhury, Power graphs: A survey. Electron. J. Graph Theory Appl 1 (2013)125-147.

 \bibitem{pj}
  G. Aalipour, S. Akbari,  P. J. Cameron,  R.  Nikandish,   F.  Shaveisi, On the structure of  the power graph and the enhanced power graph of a group. arXiv: 1603.04337v1(2016)[math. CO].


\bibitem{BB1}
 A. K. Bhuniya, S.  Bera,  On some characterizations of strong power graphs of finite groups, Spec. Matrices 4(2016) 121-129.

\bibitem{Bon}
 J. A. Bondy,  U. S. R. Murty,  Graph theory with applications, Elsevier, (1977).

\bibitem{Char}
 G. Chartrand, P. Zhang,   Introduction to graph theory, Tata McGraw-Hill, (2006).

 \bibitem{gorestin}
 D. Gorestein, Finite groups, New York, Harper and Row,  publishers, (1968).

\bibitem{Cameran}
   P. J. Cameron,  S.   Ghosh,  The power graph of a finite group. Discrete Math 311(2011)220-1222.

\bibitem{Cameran 2}
 P. J.   Cameron,   The power graph of a finite group, II. J.Group Theory 13(6)(2010)779-783.

 \bibitem{dostabadi}
  A. Doostabadi,   A. Erfanian,  A. Jafarzadeh,   Some results on the power graph of groups, The 44 th Annual Irnian Mathematics Conferance, Ferdowsi University of Mashhad, Iran (2013)27-30.

\bibitem{sen}
I. Chakrabarty, S. Ghosh, M. K. Sen,  Undirected power graphs of semigroups, Semigroup
Forum 78(2009)410-426.


\bibitem{perfect}
 N. Chudnovsky,  N. Robertson, P. Seymour,  R. thomas,  The strong perfect graph theorem, Ann. Math.164(2006)51-229.

 \bibitem{Godsil}
  C. Godsil,  G. Royle,  Algebraic Graph Theory, Springer-Verlag, New York Inc, (2001).


\bibitem{Galian}
 J. A. Gallian,  Contemporary Abstract Algebra, Narosa Publising House, (1999).

\bibitem{Hunger}
 T. W. Hungerford,   Algebra, Gratuets Text in Mathematics, New
York(NY), Springer-Verlag, 73(1974).

\bibitem{K}
  A. V. Kelarev,  S. J. Quinn,   Directed graph and combinatorial properties of semigroups,  J. aigebra 251(2002)16-26.

 \bibitem{Shi}
  A. R. Moghaddanfar,   S. Rahbariyan,  W. J. Shi,   Certain properties of the power graph associated with a finite group, arXiv: 1310.2032v1(2013)[math. GR].

 \bibitem{rotman}
  J. J. Rotman,  Advanced Modern Algebra, American Mathematical Society, (2010).

\bibitem{r}
 S. P.  Redmond,  An ideal-based zero divisor graph of acommutative ring, Communication in algebra 31(2003)4425-4443.

\bibitem{sing}
 G. Singh,  K. Manilal,  Some Generalities on Power Graphs  and Strong Power
Graphs,  Int. J. Contemp. Math Sciences 5(55)(2010)2723-2730.

\bibitem{Chelvam}
 T.  Tamizh Chelvam,  M.  Sattanathan,    Power graph of finite abelian groups,  Algebra and Discrete Mathematics 16(1) (2013)33-41.


\bibitem{w}
D. B.   West,  Introduction to Graph theory, 2nd ed. pearson education, (2001).



\end{thebibliography}

\end{document}